\documentclass[11pt,a4paper]{article}
\usepackage[utf8]{inputenc}
\usepackage[german,english]{babel}
\usepackage{amsmath,amssymb,amsthm}
\usepackage{geometry}
\newcommand{\R}{\mathbb{R}}
\newcommand{\Bloc}{BV_{\mathrm{loc}}}
\newcommand{\Wloc}{W_{\mathrm{loc}}}
\newcommand{\Lloc}{L_{\mathrm{loc}}}

\newcommand{\Lmeasure}{\mathcal{L}}

\newcommand{\nablaA}{\nabla^a} 
\newcommand{\nablaS}{\nabla^s} 
\newcommand{\nablaC}{\nabla^c} 
\newcommand{\nablaJ}{\nabla^j} 

\newtheorem{theorem}{Theorem}[section]
\newtheorem{definition}[theorem]{Definition}
\newtheorem{remark}[theorem]{Remark}
\newtheorem{corollary}[theorem]{Corollary}
\newtheorem{lemma}[theorem]{Lemma}

\numberwithin{equation}{section}

\begin{document}

\begin{center}
    {\LARGE \bf The Liouville property for entire local minimizers from the space $BV_{\mathrm{loc}}(\R^2,\R^N)$ of variational integrals with linear growth} \\[0.5cm]
    {\large Michael Bildhauer \qquad Martin Fuchs}
\end{center}

\vspace*{0.5cm}

\noindent {\bf Abstract.} We consider functions $u : \R^2 \rightarrow \R^N$ from the class $\Bloc$ which locally minimize the relaxed version of a variational integral with a strictly convex density being in addition of linear growth. It turns out that the boundedness of entire local minimizers $u$ implies their constancy.

\vspace*{0.3cm}

\vspace*{0.1cm}

\section{Introduction}

For domains $\Omega \subset \R^2$ and functions $u : \R^2 \rightarrow \R^N$, $N \ge 1$, we look at the variational integral
\begin{equation}
    J[u, \Omega] := \int_{\Omega} f(\nabla u) \, \mathrm{d}x,
\end{equation}
where the density $f : \R^{2N} \rightarrow [0,\infty)$, $\R^{2N}$ the space of $(2 \times N)$-matrices, is assumed to be strictly convex and of linear growth in the following sense: it holds
\begin{equation}
    a|p| - b \le f(p) \le A|p| + B, \quad p \in \R^{2N},
\end{equation}
\begin{equation}
    D^2 f(p)(q,q) > 0, \quad p,q \in \R^{2N} - \{0\},
\end{equation}
and
\begin{equation}
    \nabla f(0) = 0, \quad |D^2 f| \in L^\infty(\R^{2N}),
\end{equation}
where $a, b, A, B$ denote positive constants. Clearly (1.1) is well-defined for functions $u$ from the local Sobolev space $\Wloc^{1,1}(\R^2, \R^N)$ (see, e.g., [1] for a definition), however, a more natural class for the functional $J$ and its lower semicontinuous relaxation $\hat{J}$ is the class $\Bloc(\R^2, \R^N)$ consisting of all functions $u \in \Lloc^1(\R^2, \R^N)$ with locally finite total variation introduced for instance in [2], [3] and [4], where the interested reader will also find the following technical details: given a function $u \in \Bloc(\R^2, \R^N)$ we decompose (compare [3], Theorem 2.22 on page 54) the vector measure $\nabla u$ with respect to Lebesgue's measure $\Lmeasure^2$, i.e. we write
\begin{equation}
    \nabla u = \nablaA u \, \Lmeasure^2 + \nablaS u
\end{equation}
with $\nablaA u \in \Lloc^1(\R^2, \R^{2N})$ and $\nablaS u$ denoting the singular part, which takes the form
\begin{equation}
    \nablaS u = \xi |\nablaS u|
\end{equation}
with $\xi : \R^2 \rightarrow \R^{2N}$ locally integrable with respect to the measure $|\nablaS u|$, and even more precisely it holds $|\xi| = 1$ $|\nablaS u|$-almost everywhere. Using the common symbol $\frac{\nablaS u}{|\nablaS u|}$ in place of $\xi$, we get from (1.5) and (1.6) the familiar formula
\begin{equation}
    \nabla u = \nablaA u \, \Lmeasure^2 + \frac{\nablaS u}{|\nablaS u|} |\nablaS u|.
\end{equation}
Finally, quoting Theorem 5.47 on page 304 in [3], we see that a natural extension of the functional $J$ from (1.1) to the space $\Bloc(\R^2, \R^N)$ is given by
\begin{equation}
    \hat{J}[u, \Omega] := \int_{\Omega} f(\nablaA u) \, \mathrm{d}x + \int_{\Omega} f_\infty \left( \frac{\nablaS u}{|\nablaS u|} \right) \mathrm{d}|\nablaS u|,
\end{equation}
where
$$ f_\infty(p) := \lim_{t \rightarrow \infty} \frac{1}{t} f(tp), \quad p \in \R^{2N}, $$
is the recession function of $f$ and $\Omega$ denotes a planar domain. We introduce the following notation:

\begin{definition}
A function $u \in \Bloc(\R^2, \R^N)$ is an entire local minimizer of the functional $\hat{J}$ from (1.8), iff
\begin{equation}
    \hat{J}[u, \Omega] \le \hat{J}[v, \Omega]
\end{equation}
holds for all bounded domains $\Omega \subset \R^2$ and any $v \in \Bloc(\R^2, \R^N)$ such that $u = v$ outside of a compact subset of $\Omega$.
\end{definition}

It turns out that the Liouville property is true for entire local $\hat{J}$-minimizers $u$, more precisely we have:

\begin{theorem}
Let (1.2)--(1.4) hold for the density $f$ and consider an entire local $\hat{J}$-minimizer $u \in \Bloc(\R^2, \R^N)$ as in Definition 1.1. Then, if $u$ is bounded, the constancy of $u$ follows.
\end{theorem}

\begin{remark}
The Liouville property for entire local minimizers of variational integrals or for entire solutions of various types of elliptic equations has been investigated during the past decades by many authors. For an overview we refer to the nice work [5] or to the references given in Section 1 of [6].
\end{remark}

\begin{remark}
Since the density $f$ is of linear growth, one might even ask, if the Bernstein property holds for entire local minimizers $u$ under additional assumptions imposed on $f$, which means that without further hypotheses concerning $u$ the function $u$ must be affine. Actually we have no idea, if this can be achieved, or if non-affine counterexamples exist.
\end{remark}

Going through the proof of Theorem 1.2 presented in Section 2 we deduce:

\begin{corollary}
Under the assumptions (1.2), (1.3) and (1.4) concerning the density $f$ consider an entire local $\hat{J}$-minimizer $u \in \Bloc(\R^2, \R^N)$ in the sense of Definition 1.1. Then, if $\nablaA u = 0$ holds, $u$ must be a constant function.
\end{corollary}

\noindent In fact, from equation (2.1) combined with $\nabla f(\nablaA u) = \nabla f(0) = 0$ (see (1.4)) we deduce
$$ \int_{\Omega} \eta \, f_\infty \left( \frac{\nablaS u}{|\nablaS u|} \right) \mathrm{d}|\nablaS u| = 0 $$
valid for any open set $\Omega \subset \R^2$ and all $\eta \in C_0^1(\Omega)$ yielding $\nablaS u = 0$ and in conclusion $\nabla u = 0$.

\vspace*{0.3cm}
Our paper is organized as follows: Section 2 presents the proof of Theorem 1.2, and from the methods used there we obtain a kind of uniqueness result for entire local minimizers (see Theorem 2.1). Section 3 discusses the Liouville property under one-sided conditions.

\section{Proof of Theorem 1.2}

Let assumptions (1.2)--(1.4) hold for the density $f$ and consider a local $\hat{J}$-minimizer $u \in \Bloc(\R^2, \R^N)$ in the sense of (1.9). We choose
$$ \Omega := B_{2R} := \{ x \in \R^2 : |x| < 2R \}, \quad R > 0, $$
$\eta \in C_0^1(\Omega)$, $\varepsilon > 0$, and define $v := u(1 + \varepsilon \eta)$. It holds $v \in \Bloc(\R^2, \R^N)$, $\mathrm{spt}(u - v) \subset\subset \Omega$ and
\begin{align*}
\nablaA v &= (1 + \varepsilon \eta) \nablaA u + \varepsilon u \otimes \nabla \eta, \\
\nablaS v &= (1 + \varepsilon \eta) \nablaS u,
\end{align*}
$u \otimes \nabla \eta$ denoting the tensor product. Using (1.8) we obtain from (1.9)
\begin{align*}
    \int_{\Omega} f(\nablaA u) \, \mathrm{d}x &- \int_{\Omega} f(\nablaA u + \varepsilon [ \eta \nablaA u + u \otimes \nabla \eta ]) \, \mathrm{d}x \\
    &+ \int_{\Omega} f_\infty\left(\frac{\nablaS u}{|\nablaS u|}\right)\mathrm{d}|\nablaS u| - \int_{\Omega} (1 + \varepsilon \eta) f_\infty\left(\frac{\nablaS u}{|\nablaS u|}\right)\mathrm{d}|\nablaS u| \le 0.
\end{align*}
Of course this inequality is true only in the case that $1 + \varepsilon \eta \ge 0$ holds, since then we have $|(1 + \varepsilon \eta) \nablaS u| = (1 + \varepsilon \eta)|\nablaS u|$. However, the latter inequality is true for any $\varepsilon \in (0, \varepsilon(\eta)]$ with sufficiently small $\varepsilon(\eta)$. Dividing by $\varepsilon > 0$ and passing to the limit $\varepsilon \rightarrow 0$ it follows ($:$ denoting the scalar product of matrices)
$$ - \int_{\Omega} \nabla f(\nablaA u) : [ \eta \nablaA u + u \otimes \nabla \eta ] \, \mathrm{d}x - \int_{\Omega} \eta \, f_\infty\left(\frac{\nablaS u}{|\nablaS u|}\right)\mathrm{d}|\nablaS u| \le 0, $$
and since the sign of $\eta \in C_0^1(\Omega)$ is arbitrary, we get
\begin{equation}
    \int_{\Omega} \nabla f(\nablaA u) : [ \eta \nablaA u + u \otimes \nabla \eta ] \, \mathrm{d}x + \int_{\Omega} \eta \, f_\infty\left(\frac{\nablaS u}{|\nablaS u|}\right)\mathrm{d}|\nablaS u| = 0
\end{equation}
valid for any $\eta \in C_0^1(\Omega)$. Let us choose $\eta$ such that $0 \le \eta \le 1$ together with compact support in $B_{2R}$ and additionally $\eta = 1$ on $B_R$ and $|\nabla \eta| \le \frac{c}{R}$ for an absolute constant. We replace $\eta$ by $\eta^2$ in (2.1) and find
\begin{equation}
    \int_{B_{2R}} \eta^2 \nabla f(\nablaA u) : \nablaA u \, \mathrm{d}x + \int_{B_{2R}} \eta^2 f_\infty\left(\frac{\nablaS u}{|\nablaS u|}\right)\mathrm{d}|\nablaS u| = -2 \int_{B_{2R}} \eta \nabla f(\nablaA u) : (u \otimes \nabla \eta) \, \mathrm{d}x.
\end{equation}
Recalling (1.4) and introducing the parameter dependent bilinear form
$$ B_{t,x}(p,q) := D^2 f(t \nablaA u(x))(p,q),\quad 0 \le t \le 1,  x \in B_{2R},  p,q \in \R^{2N}, $$equation (2.2) implies \begin{align}
\int_{B_{2R}} \int_{0}^{1} \eta^2 B_{t,x}(\nablaA u, \nablaA u)  \mathrm{d}t  \mathrm{d}x &+ \int_{B_{2R}} \eta^2 f_\infty\left(\frac{\nablaS u}{|\nablaS u|}\right)\mathrm{d}|\nablaS u| \nonumber \\
&= -2 \int_{B_{2R}} \int_{0}^{1} B_{t,x}(\eta \nablaA u, u \otimes \nabla \eta)  \mathrm{d}t  \mathrm{d}x.
\end{align}Recalling (1.3), the Cauchy-Schwarz inequality applies to $B_{t,x}$, which means$$ |B_{t,x}(\eta \nablaA u, u \otimes \nabla \eta)| \le B_{t,x}(\eta \nablaA u, \eta \nablaA u)^{1/2} B_{t,x}(u \otimes \nabla \eta, u \otimes \nabla \eta)^{1/2}, $$and after an application of Hölder's inequality to the quantity$$ \int_{\mathrm{spt}(\nabla \eta)} \int_{0}^{1} |B_{t,x}(\eta \nablaA u, u \otimes \nabla \eta)|  \mathrm{d}t  \mathrm{d}x $$(having inserted the latter estimate), we get from (2.3)\begin{align}
\int_{B_{2R}} \int_{0}^{1} \eta^2 B_{t,x}(\nablaA u, \nablaA u)  \mathrm{d}t  \mathrm{d}x &+ \int_{B_{2R}} \eta^2 f_\infty\left(\frac{\nablaS u}{|\nablaS u|}\right)\mathrm{d}|\nablaS u| \nonumber \\
&\le c \left( \int_{\mathrm{spt}(\nabla \eta)} \int_{0}^{1} \eta^2 B_{t,x}(\nablaA u, \nablaA u)  \mathrm{d}t  \mathrm{d}x \right)^{1/2} \nonumber \\
&\quad \cdot \left( \int_{\mathrm{spt}(\nabla \eta)} \int_{0}^{1} B_{t,x}(u \otimes \nabla \eta, u \otimes \nabla \eta)  \mathrm{d}t  \mathrm{d}x \right)^{1/2}.\end{align}Recalling the boundedness of $|D^2 f|$ (see (1.4)) and observing that in the situation of Theorem 1.2 we assume that $u \in L^\infty(\R^2, \R^N)$, we see that$$ \int_{\mathrm{spt}(\nabla \eta)} \int_{0}^{1} B_{t,x}(u \otimes \nabla \eta, u \otimes \nabla \eta)  \mathrm{d}t  \mathrm{d}x $$can be bounded by a finite constant independent of $R$. Applying Young's inequality to the right-hand side of (2.4), we therefore deduce for a finite constant $K$\begin{equation}\int_{B_R} \int_{0}^{1} B_{t,x}(\nablaA u, \nablaA u)  \mathrm{d}t  \mathrm{d}x + \int_{B_R} f_\infty\left(\frac{\nablaS u}{|\nablaS u|}\right)\mathrm{d}|\nablaS u| \le K\end{equation}valid for any $R > 0$. Here we note that $f_\infty \ge 0$, which follows from \begin{align*}
\frac{1}{t} f(tp) &\ge \frac{1}{t} [f(p) + \nabla f(p) : (tp - p)] \\
&\rightarrow p : \nabla f(p) \quad \text{as } t \rightarrow \infty
\end{align*}for any $p \in \R^{2N}$, and we have $p : \nabla f(p) \ge 0$ by (1.3) and (1.4). Returning to (2.4) and applying (2.5) with the consequence that$$ \int_{B_{2R}-B_R} \int_{0}^{1} B_{t,x}(\nablaA u, \nablaA u)  \mathrm{d}t  \mathrm{d}x \rightarrow 0 \quad \text{as } R \rightarrow \infty, $$we end up with\begin{equation}\int_{\R^2} \int_{0}^{1} B_{t,x}(\nablaA u, \nablaA u)  \mathrm{d}t  \mathrm{d}x + \int_{\R^2} f_\infty\left(\frac{\nablaS u}{|\nablaS u|}\right)\mathrm{d}|\nablaS u| = 0.\end{equation}Finally, we use property (1.3) and obtain from (2.6) that $\nablaA u$ must vanish, whereas from$$ \int_{\R^2} f_\infty\left(\frac{\nablaS u}{|\nablaS u|}\right)\mathrm{d}|\nablaS u| = 0 $$combined with the foregoing estimate (observe (1.4))$$ f_\infty(p) \ge p \cdot \nabla f(p) = \int_{0}^{1} D^2 f(tp)(p,p) \mathrm{d}t $$implying that $f_\infty(p) > 0$ for $p \in \R^{2N} - \{0\}$ (see (1.3)), it follows that also the singular part $\nablaS u$ is identically zero. Altogether it is shown that $\nabla u = 0$ and the constancy of $u$ follows. $\square$\vspace*{0.3cm}\\An inspection of the above arguments shows\begin{theorem}Let the density $f$ satisfy (1.2)--(1.4) and consider entire local $\hat{J}$-minimizers $u_1, u_2 \in \Bloc(\R^2, \R^N)$ in the sense of Definition 1.1. Suppose that $u_1 - u_2 \in L^\infty(\R^2, \R^N)$ and that $u_i \in C^2(\R^2, \R^N)$ for at least one index $i$. Then we have the constancy of $u_1 - u_2$, provided $f(p) = g(|p|)$ with $g \in C^2([0,\infty))$.\end{theorem}Clearly we obtain Theorem 1.2 by just letting $u_1 = u$ and $u_2 = 0$. In the general case we assume $u_2 \in C^2(\R^2, \R^N)$ and observe first\begin{equation}\hat{J}[u_1, \Omega] \le \hat{J}[u_1 + \varepsilon \eta (u_2 - u_1), \Omega],\end{equation}where $\Omega := B_{2R}$, $\eta \in C_0^1(\Omega)$ and $\varepsilon > 0$ is chosen to satisfy $1 - \varepsilon \eta \ge 0$. We have \begin{align*}
\nablaA [u_1 + \varepsilon \eta (u_2 - u_1)] &= \nablaA u_1 + \varepsilon \eta (\nabla  u_2 - \nablaA u_1) + \varepsilon (u_2 - u_1) \otimes \nabla \eta, \\
\nablaS [u_1 + \varepsilon \eta (u_2 - u_1)] &= (1 - \varepsilon \eta) \nablaS u_1.
\end{align*}From (2.7) we get along the lines leading to (2.1)\begin{align}
- \int_{\Omega} \nabla f(\nablaA u_1) &: [ \eta (\nabla u_2 - \nablaA u_1) + (u_2 - u_1) \otimes \nabla \eta ]  \mathrm{d}x \nonumber \\
&+ \int_{\Omega} \eta f_\infty\left(\frac{\nablaS u_1}{|\nablaS u_1|}\right)\mathrm{d}|\nablaS u_1| = 0.
\end{align} Second we use\begin{equation}\hat{J}[u_2, \Omega] = J[u_2, \Omega] \le \hat{J}[u_2 + \varepsilon \eta (u_2 - u_1), \Omega]\end{equation}and obtain from (2.9) after dividing by $\varepsilon > 0$ and passing to the limit $\varepsilon \rightarrow 0$ \begin{align}
- \int_{\Omega} \nabla f(\nabla u_2) &: [ \eta (-\nablaA u_1 + \nabla u_2) + (u_1 - u_2) \otimes \nabla \eta ] \, \mathrm{d}x \nonumber \\
&\le \int_{\Omega} f_\infty(\xi) \, \mathrm{d}|\nablaS( \eta u_1)|, \quad \xi := \frac{\nablaS(-\eta u_1)}{|\nablaS(-\eta u_1)|}.
\end{align}Since $f$ is an even function, the right-hand side of (2.10) equals$$ \int_{\Omega} f_\infty\left(\frac{\nablaS u_1}{|\nablaS u_1|}\right) |\eta|  \mathrm{d}|\nablaS u_1|, $$and if we replace $\eta$ by $-\eta$ in (2.10), we find for any $\eta \in C_0^1(\Omega)$ \begin{align}
\int_{\Omega} \nabla f(\nabla u_2) &: [ \eta (\nabla u_2 - \nablaA u_1) + (u_2 - u_1) \otimes \nabla \eta ] \, \mathrm{d}x \nonumber \\
&\le \int_{\Omega} f_\infty\left(\frac{\nablaS u_1}{|\nablaS u_1|}\right) |\eta| \, \mathrm{d}|\nablaS u_1|.
\end{align}Consider now $\eta$ as specified after (2.1), replace $\eta$ by $\eta^2$ in (2.8) and (2.11) and finally add (2.8) to (2.11) with the result \begin{align}
\int_{\Omega} (\nabla f(\nabla u_2) - \nabla f(\nablaA u_1)) &: (\nabla u_2 - \nablaA u_1) \eta^2 \, \mathrm{d}x \nonumber \\
&+ \int_{\Omega} (\nabla f(\nabla u_2) - \nabla f(\nablaA u_1)) &: (u_2 - u_1) \otimes \nabla \eta^2 \, \mathrm{d}x \le 0.
\end{align}We replace the bilinear form $B_{t,x}(p,q)$ by ($0 \le t \le 1, x \in \Omega$)$$ \widetilde{B}_{t,x}(p,q) := D^2 f(\nablaA u_1(x) + t(\nabla u_2(x) - \nablaA u_1(x)))(p,q) $$and proceed as done after (2.3) starting from (2.12) using the boundedness of $u_2 - u_1$ and get $\nabla u_2 = \nablaA u_1$. Up to now we have not used the structure condition $f(p) = g(|p|)$, and the assumption $u_2 \in C^2(\R^2, \R^N)$ is not needed for proving $\nablaA u_1 = \nabla u_2$. Both hypotheses will enter the discussion whether $\nablaS u_1 = 0$ holds. For notational simplicity we consider the case $N = 1$ and apply equation (2.1) for the choice $u := u_1$ replacing $\nablaA u_1$ by $\nabla u_2 \in C^1(\R^2, \R^2)$ with the result\begin{equation}\int_{\Omega} \nabla f(\nabla u_2) \cdot \nabla u_2  \eta \mathrm{d}x + \int_{\Omega} \nabla f(\nabla u_2) \cdot \nabla \eta  u_1  \mathrm{d}x + \int_{\Omega} \eta  f_\infty(\xi_1)  \mathrm{d}|\nablaS u_1| = 0,\end{equation}where we have set $\xi_1 := \frac{\nablaS u_1}{|\nablaS u_1|}$ and $\Omega, \eta$ are as in (2.1). It holds\begin{align*}
f_\infty(p) &= g'(\infty)|p|, \quad p \in \mathbb{R}^2, \\
g'(\infty) &:= \lim_{t \rightarrow \infty} g'(t) > 0.
\end{align*} hence we obtain from (2.13), using the equation$$ \mathrm{div}(\nabla f(\nabla u_2)) = 0 $$valid for the function $u_2$, the identity \begin{align}
0 &= \int_{\Omega} \eta g'(\infty) \, \mathrm{d}|\nablaS u_1| + \int_{\Omega} \operatorname{div}[\nabla f(\nabla u_2) \eta] u_1 \, \mathrm{d}x + \int_{\Omega} \nabla f(\nabla u_2) \cdot \nabla u_2 \eta \, \mathrm{d}x \nonumber \\
&= \int_{\Omega} \eta g'(\infty) \, \mathrm{d}|\nablaS u_1| - \int_{\Omega} \nabla f(\nabla u_2) \eta \cdot \mathrm{d}\nabla u_1 + \int_{\Omega} \nabla f(\nabla u_2) \cdot \nabla u_2 \eta \, \mathrm{d}x,
\end{align}where we have made use of the Gauss theorem in the space $BV$ (see, e.g., [2], Theorem 2.10 on page 34). It holds (again on account of $\nablaA u_1 = \nabla u_2$)\begin{align*}
\int_{\Omega} \nabla f(\nabla u_2) \eta \cdot \mathrm{d}\nabla u_1 &= \int_{\Omega} \nabla f(\nabla u_2) \cdot \nablaA u_1 \eta \, \mathrm{d}x + \int_{\Omega} \nabla f(\nabla u_2) \cdot \xi_1 \eta \, \mathrm{d}|\nablaS u_1| \\
&= \int_{\Omega} \nabla f(\nabla u_2) \cdot \nabla u_2 \eta \, \mathrm{d}x + \int_{\Omega} \nabla f(\nabla u_2) \cdot \xi_1 \eta \, \mathrm{d}|\nablaS u_1|,
\end{align*}
hence (2.14) finally gives
\[\int_{\Omega} \eta g'(\infty) \, \mathrm{d}|\nablaS u_1| = \int_{\Omega} \eta \nabla f(\nabla u_2) \cdot \xi_1 \, \mathrm{d}|\nablaS u_1|\] for any domain $\Omega$ in $\R^2$ and functions $\eta \in C_0^1(\Omega)$, and therefore $g'(\infty) = \nabla f(\nabla u_2) \cdot \xi_1$ $|\nablaS u_1|$-almost everywhere. Suppose that there exists a set $E$ such that $|\nablaS u_1|(E) > 0$. If $\nabla u_2 = 0$ $|\nablaS u_1|$-almost everywhere on $E$, we obtain the contradiction $g'(\infty) = 0$. For points $x \in E$ with $\nabla u_2(x) \neq 0$ we have$$ g'(\infty) = g'(|\nabla u_2(x)|) \xi_1(x) \cdot \frac{\nabla u_2(x)}{|\nabla u_2(x)|}, $$and in conclusion $g'(\infty) \le g'(|\nabla u_2(x)|)$ again leading to a contradiction. This finally shows that $u_1 - u_2$ is constant. $\square$\begin{remark}For the interested reader we note that the result of Theorem 2.1 is not necessarily based on the structure condition $f(p) = g(|p|)$, $p \in \R^{2N}$. Actually it holds, if the density $f$ is an even function with the properties (1.2)--(1.4). In fact, exactly the same calculations (see (2.14) and compare the discussion after this equation) will lead to the identity\begin{equation}f_\infty\left(\frac{\nablaS u_1}{|\nablaS u_1|}\right) = \nabla f(\nabla u_2) \cdot \frac{\nablaS u_1}{|\nablaS u_1|}\end{equation}for $|\nablaS u_1|$-almost all points. Now, from (1.2)--(1.4) we get the strict inequality$$ f_\infty(p) > \nabla f(q) \cdot p $$valid for all choices of $p$ and $q$ with $p \neq 0$, hence (2.15) shows that $\nablaS u_1$ must vanish.\end{remark}\begin{remark}Let us consider the more geometric case that $N = 1$. Clearly any affine function $v : \R^2 \rightarrow \R$ is an entire local minimizer, and Theorem 2.1 implies that in any $L^\infty$-neighborhood of such a plane no other surface being the graph of an entire local $BV$-minimizer $\R^2 \rightarrow \R$ exists.\end{remark}\section{Some extensions: one-sided conditions}It turns out that in the scalar case a one-sided condition is sufficient for the Liouville property, more precisely, we have:\begin{theorem}Let $N = 1$ and consider a density $f \in C^2(\R^2)$ with (1.2)--(1.4) being of the type $f(p) = g(|p|)$, $p \in \R^2$, for a function $g \in C^2([0,\infty))$. If $u \in \Bloc(\R^2)$ denotes an entire local $\hat{J}$-minimizer in the sense of Definition 1.1 being bounded from one side, then $u$ is constant.\end{theorem}\begin{remark}We leave open to decide, whether the structure condition for $f$ is really necessary.\end{remark}\noindent {\bf Proof of Theorem 3.1:} Without loss of generality we may assume that $u \ge 1$ $\Lmeasure^2$-almost everywhere. As in Section 2 we let $\Omega := B_{2R}$ for some $R > 0$ and choose $\eta \in C_0^1(\Omega)$ and $\varepsilon > 0$. With $\Phi : [1,\infty) \rightarrow \R$, $\Phi(t) := \frac{1}{t}$, we define$$ u_\varepsilon := u + \varepsilon \eta \Phi(u) $$and note that on account of Theorem 3.96, p.189, in [3] it holds $\Phi(u) \in BV(\Omega)$. The local minimality of $u$ implies\begin{equation}\int_{\Omega} f(\nablaA u)  \mathrm{d}x - \int_{\Omega} f(\nablaA u_\varepsilon)  \mathrm{d}x + \int_{\Omega} f_\infty\left(\frac{\nablaS u}{|\nablaS u|}\right)\mathrm{d}|\nablaS u| - \int_{\Omega} f_\infty\left(\frac{\nablaS u_\varepsilon}{|\nablaS u_\varepsilon|}\right)\mathrm{d}|\nablaS u_\varepsilon| \le 0.\end{equation}We have (see equation (3.95), p.189, in [3])\begin{align*}
\nablaA u_\varepsilon &= \nablaA u + \varepsilon \eta \Phi'(u) \nablaA u + \varepsilon \Phi(u) \nabla \eta, \\
\nablaJ u_\varepsilon &= \nablaJ u + \varepsilon \eta \nablaJ \Phi(u) = (u^+ - u^-)\nu_u \mathcal{H}^1 \lfloor J_u + \varepsilon \eta (\Phi(u^+) - \Phi(u^-))\nu_u \mathcal{H}^1 \lfloor J_u, \\
\nablaC u_\varepsilon &= \nablaC u + \varepsilon \eta \Phi'(\widetilde{u}) \nablaC u,
\end{align*} where we also quoted formula (3.90), p.184, from [3] and where $\widetilde{u}$ is the representative of $u$ introduced in Definition 3.63 of [3] on page 160. Since we require $f(p) = g(|p|)$, it holds$$ f_\infty(p) = g'(\infty)|p|, \quad p \in \R^2, $$and of course we may assume that $g'(\infty) := \lim_{t \rightarrow \infty} g'(t) = 1$. This implies\begin{align*}
\int_{\Omega} f_\infty\left(\frac{\nablaS u}{|\nablaS u|}\right)\mathrm{d}|\nablaS u| &- \int_{\Omega} f_\infty\left(\frac{\nablaS u_\varepsilon}{|\nablaS u_\varepsilon|}\right)\mathrm{d}|\nablaS u_\varepsilon| \\
&= \int_{\Omega} \mathrm{d}|\nablaS u| - \int_{\Omega} \mathrm{d}|\nablaS u_\varepsilon| \\
&= \int_{\Omega \cap J_u} |u^+ - u^-| \, \mathrm{d}\mathcal{H}^1 - \int_{\Omega \cap J_u} |u^+ - u^-| \left| 1 + \varepsilon \eta \frac{\Phi(u^+) - \Phi(u^-)}{u^+ - u^-} \right| \, \mathrm{d}\mathcal{H}^1 \\
&\quad + \int_{\Omega} \mathrm{d}|\nablaC u| - \int_{\Omega} \mathrm{d}|(1 + \varepsilon \eta \Phi'(\widetilde{u}))\nablaC u|.
\end{align*} From the properties of $\Phi$ we deduce that for $0 < \varepsilon \le \varepsilon(\eta)$ we can arrange that$$ 1 + \varepsilon \eta \frac{\Phi(u^+) - \Phi(u^-)}{u^+ - u^-} $$and$$ 1 + \varepsilon \eta \Phi'(\widetilde{u}) $$take values in $[1/2, 1]$ and from Lemma 3.3 below it follows\begin{align*}
\int_{\Omega} f_\infty\left(\frac{\nablaS u}{|\nablaS u|}\right)\mathrm{d}|\nablaS u| &- \int_{\Omega} f_\infty\left(\frac{\nablaS u_\varepsilon}{|\nablaS u_\varepsilon|}\right)\mathrm{d}|\nablaS u_\varepsilon| \\
&= -\varepsilon \int_{\Omega \cap J_u} \eta |u^+ - u^-| \frac{\Phi(u^+) - \Phi(u^-)}{u^+ - u^-} \, \mathrm{d}\mathcal{H}^1 - \varepsilon \int_{\Omega} \eta \Phi'(\widetilde{u}) \, \mathrm{d}|\nablaC u|.
\end{align*}Inserting this equation in (3.1), dividing by $\varepsilon > 0$ and passing to the limit $\varepsilon \rightarrow 0$, we find \begin{align*}
&- \int_{\Omega} \nabla f(\nablaA u) \cdot (\eta \Phi'(u) \nablaA u + \Phi(u) \nabla \eta) \, \mathrm{d}x \\
&- \int_{\Omega \cap J_u} \eta |u^+ - u^-| \frac{\Phi(u^+) - \Phi(u^-)}{u^+ - u^-} \, \mathrm{d}\mathcal{H}^1 - \int_{\Omega} \eta \Phi'(\widetilde{u}) \, \mathrm{d}|\nablaC u| \le 0,
\end{align*}and since $\eta$ is an arbitrary function from $C_0^1(\Omega)$, we end up with the equation \begin{align}
\int_{\Omega} \eta u^{-2} \nabla f(\nablaA u) \cdot \nablaA u \, \mathrm{d}x &- \int_{\Omega} \nabla f(\nablaA u) \cdot \frac{1}{u} \nabla \eta \, \mathrm{d}x \nonumber \\
&\quad + \int_{\Omega} \eta \widetilde{u}^{-2} \, \mathrm{d}|\nablaC u| + \int_{\Omega \cap J_u} \frac{1}{u^+ u^-} |u^+ - u^-| \eta \, \mathrm{d}\mathcal{H}^1 = 0
\end{align}corresponding to (2.1). With minor changes we can proceed as before starting from the appropriate version of (2.3), where now in the first integral on the left-hand side of (2.3) the factor $u^{-2}$ occurs in the integrand, etc. This gives first that $\nablaA u$ must vanish, and second -- after returning to (3.2) -- we also get $\nablaC u = 0$ and $\nablaJ u = 0$ on the whole plane. This proves Theorem 3.1. $\square$\vspace*{0.2cm}\\During our calculations we made use of\begin{lemma}Let $\mu$ be a $\R^m$-valued measure on the measure space $(X, \Sigma)$ such that $|\mu|(X) < \infty$, where $|\mu|$ is the total variation of $\mu$ (see [3], Definition 1.4, p.3). Let $g \in L^1(X, |\mu|)$ such that $g \ge 0$ $|\mu|$-almost everywhere and define the $\R^m$-valued measure $\widetilde{\mu} := \mu \lfloor g$, i.e. $\widetilde{\mu}(B) = \int_B g  \mathrm{d}\mu$, $B \in \Sigma$. Then it holds$$ |\widetilde{\mu}| = |\mu| \lfloor g. $$\end{lemma}\begin{proof}We quote [3], Corollary 1.29, p.14, to get $\mu = |\mu| \lfloor E$, $\widetilde{\mu} = |\widetilde{\mu}| \lfloor \widetilde{E}$ with $S^{m-1}$-valued functions $E, \widetilde{E}$. Moreover, it holds $|\widetilde{\mu}| \ll |\mu|$: if $B \in \Sigma$ such that $|\mu|(B) = 0$, then we get$$ \widetilde{\mu}(B) = \int_B g  \mathrm{d}\mu = \int_B g E  \mathrm{d}|\mu| = 0, $$hence $\widetilde{\mu} \ll |\mu|$, and in conclusion $|\widetilde{\mu}| \ll |\mu|$. By the Radon-Nikodym theorem ([3], Theorem 1.28, p.14) we find $\Theta \in L^1(X, |\mu|)$, $\Theta \ge 0$, such that $|\widetilde{\mu}| = |\mu| \lfloor \Theta$. For any set $B$ we get:$$ \widetilde{\mu}(B) = \int_B g  \mathrm{d}\mu = \int_B g E  \mathrm{d}|\mu| $$and at the same time$$ \widetilde{\mu}(B) = \int_B \widetilde{E}  \mathrm{d}|\widetilde{\mu}| = \int_B \widetilde{E} \Theta  \mathrm{d}|\mu|, $$so that$$ g E = \widetilde{E} \Theta \quad \text{almost everywhere.} $$Recalling $|E| = 1$, $|\widetilde{E}| = 1$, $\Theta, g \ge 0$ we deduce $g \equiv \Theta$, which proves our claim. \end{proof}\vspace*{0.5cm}\noindent {\bf References}\begin{enumerate}\item[{[1] }] R.A. Adams, {\it Sobolev spaces}, Academic Press, New York 1975.\item[{[2] }] E. Giusti, {\it Minimal surfaces and functions of bounded variation}, Birkhäuser, Basel 1984.\item[{[3] }] L. Ambrosio, N. Fusco, D. Pallara, {\it Functions of bounded variation and free discontinuity problems}, Oxford University Press, New York 2000.\item[{[4] }] M. Giaquinta, G. Modica and J. Soucek, {\it Cartesian currents in the calculus of variations I}, Berlin: Springer-Verlag, 1998.\item[{[5] }] E. Mitidieri, A view on Liouville theorems in PDEs. (English. English summary) {\it Anal. Geom. Metr. Spaces} 12 (2024), no. 1, Paper No. 20240008, 21 pp.\item[{[6] }] M. Bildhauer, M. Fuchs, A sufficient condition for the Liouville property of solutions to the Euler-Lagrange equations of variational integrals with nonstandard growth in terms of the ellipticity modulus near the origin (to appear).\end{enumerate}\vspace*{0.5cm}\begin{minipage}[t]{0.48\textwidth}

\textbf{Michael Bildhauer}\\
Saarland University\\
Department of Mathematics\\
P.O. Box 15 11 50\\
66041 Saarbrücken, Germany\\
\texttt{bibi@math.uni-sb.de}

\end{minipage}
\hfill
\begin{minipage}[t]{0.48\textwidth}

\textbf{Martin Fuchs}\\
Saarland University\\
Department of Mathematics\\
P.O. Box 15 11 50\\
66041 Saarbrücken, Germany\\
\texttt{fuchs@math.uni-sb.de}

\end{minipage}\end{document}